\documentclass[10pt]{amsart} 
\usepackage{amssymb}
\usepackage[english]{babel}
\usepackage[cp1250]{inputenc}   
\usepackage[T1]{fontenc}
\usepackage{hyperref}
\usepackage[all]{xy}            
\usepackage{tikz}
\usepackage{mathtools}
\usepackage{longtable}

\newtheorem{Thm}{Theorem}[section]
\newtheorem{Prp}[Thm]{Proposition}

\theoremstyle{definition}

\theoremstyle{remark}

\numberwithin{equation}{section}

\newcommand{\N}{\mathbb{N}}

\def\Ker{\mathrm{Ker}}
\def\Im{\mathrm{Im}}
\def\Hom{\mathrm{Hom}}
\def\mod{\mathrm{mod}}      
\def\^{\textasciicircum}
\DeclareMathOperator*{\bigast}{\raisebox{-0.6ex}{\scalebox{2.2}{$\ast$}}}
\newcommand{\icases}[7]{#7\{\!\begin{smallmatrix}
                                #5#1\hfill;  & ~#5#2\hfill &\!\!\\[#6]
                                #5#3\hfill;  & ~#5#4\hfill &\!\!
                              \end{smallmatrix}}

\begin{document}
\title{Tensor, Symmetric, Exterior, and Other\\ Powers of Persistence Modules}
\author{Leon Lampret}
\address{Department of Mathematics, University of Ljubljana, Slovenia}
\email{leon.lampret@imfm.si}
\date{March 28, 2015}
\keywords{persistent homology, chain complex, tensor product, group action}
\subjclass{55N99, 55-04, 13P20, 18G35}

\begin{abstract}
We reformulate the persistent (co)homology of simplicial filtrations, viewed from a more algebraic setting, namely as the (co)homology of a chain complex of graded modules over polynomial ring $K[t]$. We also define persistent (co)homology of groups, associative algebras, Lie algebras, etc.
\par Then we obtain formulas for tensor powers $T^n(M),S^n(M),\Lambda^{\!n}(M)$ where $M$ is a persistence module. We discuss the cyclic and dihedral powers of persistence modules, and more generally quotients of $T^n(M)$ by a group action.
\end{abstract}

\maketitle

\subsection*{Results} We formulate an explicit chain complex of modules over ring $K[t]$, whose (co)homology is the persistent (co)homology of a simplicial filtration. 
\par We obtain explicit formulas for tensor powers $T^n\!M, S^n\!M, \Lambda^{\!n}\!M$, and the generalization $T^n_GM$ (special cases are the cyclic and dihedral powers), where $M\cong R^r\!\oplus\!\bigoplus_{i=1}^sR/Ra_i$. We also compute presentations of algebras $T(M), S(M), \Lambda(M)$.


\vspace{4mm}
\section{Simplicial Filtrations}
A filtration of a simplicial complex $\Delta$ is a sequence of complexes $\Delta_0,\Delta_1,\ldots,\Delta_N$, such that $\Delta_i$ is a subcomplex of $\Delta_{i+1}$ for every $i$, and $\Delta_N\!=\!\Delta$. They arise naturally in applied topology, e.g. $\Delta_i$ is the \v{C}ech complex or Vietoris-Rips complex with radius $r_i$ (so $r_0\!<\!\ldots\!<\!r_N$).
\par The desire is to know the holes of $\Delta_i$ (i.e. nonzero elements of $H_\ast(\Delta_i)$) which do not seem to go away and `persist' (i.e. are also nonzero in $H_\ast(\Delta_{i+j})$ for many $j$), as they are the characteristic features that tell us how $\Delta$ is built out of all $\Delta_i$.

\vspace{4mm}
\section{Persistence Modules}
In this section, we describe an algebraic formulation of persistent (co)homology of a simplicial filtration, invented in \cite{citeZomorodianCarlssonCTP} and \cite{citeEdelsbrunnerLetscherZomorodianTPS}, and reformulated in \cite{citeZomorodianCarlssonCPH}.\vspace{2mm}

\subsection{Gradings} A ring $R$ is \emph{graded} if it has a decomposition $R\!=\!\bigoplus_{k\in\N}\!R_k$ such that $R_k\!\cdot\!R_l\!\subseteq\!R_{k+l}$ for all $k$ and $l$.  Given a graded ring $R$, an $R$-module is \emph{graded} if it has decomposition $M\!=\!\bigoplus_{k\in\N}\!M_k$ such that $R_k\!\cdot\!M_l\!\subseteq\!M_{k+l}$ for all $k$ and $l$. For any $x\!\in\!R_k$ or $x\!\in\!M_k$ we write $\deg x\!=\!k$. Thus $R$ is itself a graded $R$-module.
\par Given graded modules $M$ and $N$, a \emph{graded morphism} is a module morphism $f\!:M\!\rightarrow\!N$ such that $f(M_k)\!\subseteq\!N_k$ for all $k$. Then $\Ker f\!=\!\bigoplus_{k\in\N}\!f^{-1}(0)\!\cap\!M_k$ and $\Im f\!=\!\bigoplus_{k\in\N}\!f(M_k)$ are graded modules. A submodule $M'$ of a graded module $M$ is a \emph{graded submodule} when $M'\!=\!\bigoplus_{k\in\N}\!M'\!\cap\!M_k$. Then the quotient $M/M'=\bigoplus_{k\in\N}\!M_k/M'\!\cap\!M_k$ is a graded module.\vspace{2mm}

\subsection{Chain Complex} Let $\Delta_\ast\!:\Delta_0\!\leq\! \Delta_1\!\leq\! \Delta_2\!\leq\! \ldots\!\leq\!\Delta_N\!=\!\Delta$ be a filtration of a finite simplicial complex and $\preceq$ a total order on the set of vertices $\Delta^{\![0]}$. Let $K$ be any field and let the univariate polynomial ring $R\!=\!K[t]$ (which is a PID by \cite[3.5.12]{citeGrilletAA}) have the standard grading $R\!=\!\bigoplus_{k\in\N}\!Kt^k$, or equivalently $\deg t\!=\!1$. The \emph{persistence complex} of $\Delta_\ast$ is the chain complex of graded $R$-modules \vspace{0pt}
\begin{longtable}[c]{c}
$P_\ast(\Delta_\ast)\!:\;\; \ldots\longrightarrow R^{(\Delta^{\![n]})} \overset{\partial_n}{\longrightarrow} R^{(\Delta^{\![n\!-\!1]})} \longrightarrow \ldots \longrightarrow  R^{(\Delta^{\![1]})} \overset{\partial_1}{\longrightarrow}  R^{(\Delta^{\![0]})}$\\[-0pt]
with $\deg\sigma\!=\!\min\{k; \sigma\!\in\!\Delta_k\}$, so $R^{(\Delta^{\![n]})}\!=\!\bigoplus_{k\geq0}\!R^{(\Delta^{\![n]}_k\setminus \Delta^{\![n]}_{k\!-\!1})}$ is the grading,\\[-0pt]
and boundary map sends $\Delta^{\![n]}\!\ni\!\sigma \longmapsto \sum_{\tau\in \Delta^{\![n\!-\!1]}}[\sigma,\tau]t^{\deg\!\sigma-\deg\!\tau}\tau$,
\end{longtable}\vspace{0pt}
\noindent where $[\sigma,\tau]=(-1)^i$ if $\sigma\!\!\setminus\!\tau$ is the $i$-th vertex of $\sigma$ w.r.t. $\preceq$, and $[\sigma,\tau]=0$ if $\tau\!\nsubseteq\!\sigma$.
\par For any $R$-module $M$, the \emph{persistent (co)homology} of $\Delta_\ast$ is $$H_\ast(P_\ast(\Delta_\ast)\!\otimes_R\!M)\text{ ~~~~~and~~~~~ }H^\ast(\Hom_R(P_\ast(\Delta_\ast),M)).\vspace{1mm}$$
Every $K$-module is a graded $R$-module, with everything concentrated in degree $0$ and $tM\!=\!0$. But the above definition applies to an arbitrary graded $R$-module $M$.\vspace{2mm}

\subsection{Example} For any set $I$, let $R^{(I)}$ denote the free module on $I$. If index set $I\!=\!\{i_1,\ldots,i_r\}$ is finite, then denote $R^{(i_1,\ldots,i_r)}\!=R^{(I)}$. Below left is an instance of a simplicial filtration, and below right is its corresponding chain complex.
$$\raisebox{-0.87\height}{\includegraphics[width=0.35\textwidth]{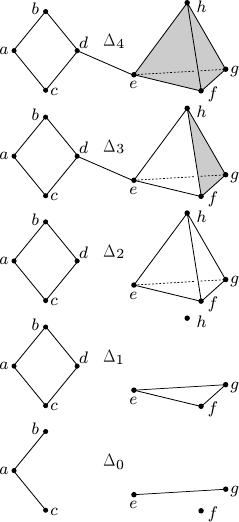}}  \hspace{23pt}
\xymatrix@R=14mm@C=8mm{&&R^{(e\!f\!g,e\!f\!h,egh)}\ar|-<<<<<<<<<<<<<<<<<<{t^2}[ldd]\ar|-<<<<<<<<<<<<<<<<<<<<{t^3}[lddd]\ar|-<<<<<<<<<<<<<<<<<<<<<<<<<<<<<{t^4}[ldddd]\ar@{}|-{\oplus}[d]\\
                                &R^{(de)}\ar|-{t^2}[ldd]\ar|-<<<<<<<<<<<<<<<<<<<<<{t^3}[lddd]\ar@{}|-{\oplus}[d]    &R^{(f\!gh)}\ar|-<<<<<<<<<{t}[ld]\ar|-<<<<<<<<<<<<<<<<{t^2}[ldd]\\
                                &R^{(eh,f\!h,gh)}\ar|-<<<<<<<<{t}[ld]\ar|-<<<<<<<<<<{t^2}[ldd]\ar@{}|-{\oplus}[d]\\
R^{(d,h)}\ar@{}|-{\oplus}[d]    &R^{(bd,cd,e\!f\!,f\!g)}\ar|-{1}[l]\ar|-{t}[ld]\ar@{}|-{\oplus}[d]\\
R^{(a,b,c,e,f\!,g)}             &R^{(ab,ac,eg)}\ar|-{1}[l]\\}$$
The following \textsc{Sage} code that uses the ring $\mathbb{Q}[t]$ and two sparse matrices\vspace{3pt}
\begin{longtable}[l]{@{\hspace{2mm}}l}
    \tt\small R.<t>=QQ[];\\
    \tt\small matrix(R,8\!,11\!,\{(0\!,0)\!:-1\!,(1\!,0)\!:1\!,(0\!,1)\!:-1\!,(2\!,1)\!:1\!,(3\!,2)\!:-1\!,(5\!,2)\!:1\!,(1\!,3)\!:-t\!,(6\!,3)\!:1\!,\\
    \tt\small                 (2\!,4)\!:-t\!,(6\!,4)\!:1\!,(3\!,5)\!:-t\!,(4\!,5)\!:t\!,(4\!,6)\!:-t\!,(5\!,6)\!:t\!,(3\!,7)\!:-t\^2,(7\!,7)\!:t,(4\!,8)\!:-t\^2,\\
    \tt\small                 (7\!,8)\!:t,(5\!,9)\!:-t\^2,(7\!,9)\!:t,(6\!,10)\!:-t\^2,(3\!,10)\!:t\^3\}).elementary\_divisors();\\
    \tt\small matrix(R,11\!,4\!,\{(9\!,0)\!:t\!,(8\!,0)\!:-t\!,(6\!,0)\!:t\^2\!,(6\!,1)\!:t\^3\!,(2\!,1)\!:-t\^4\!,(5\!,1)\!:t\^3\!,(8\!,2)\!:t\^2\!,\\
    \tt\small                 (7\!,2)\!:-t\^2\!,(5\!,2)\!:t\^3\!,(9\!,3)\!:t\^2\!,(7\!,3)\!:-t\^2\!,(2\!,3)\!:t\^4\}).elementary\_divisors();\\
\end{longtable}\vspace{2pt}
\noindent tells us that $\partial_1$ and $\partial_2$ have elementary divisors $1,1,1,1,t,t,t^3$ and $t,t^2,t^3$.
Thus the nonzero (co)homologies over a field of characteristic $0$ of this filtration are
\begin{longtable}[c]{l}
$H_0\cong R \!\oplus\! \frac{R}{Rt^3} \!\oplus\! \frac{R}{Rt} \!\oplus\! \frac{R}{Rt}$ generated by $\{a,e,f,h\}$,\\[3pt]
$H_1\cong R \!\oplus\! \frac{R}{Rt^3} \!\oplus\! \frac{R}{Rt^2} \!\oplus\! \frac{R}{Rt}$ generated by\\
~~~~~~~~{\small $\{ab\!+\!bd\!-\!cd\!-\!ac, ef\!+\!fh\!-\!eh, eg\!+\!gh\!-\!eh, fg\!+\!gh\!-\!fh\}$},\\
~~~~~~~~{with \small $ef\!+\!fg\!-\!eg\!=\! (ef\!+\!fh\!-\!eh) \!-\! (eg\!+\!gh\!-\!eh) \!+\! (fg\!+\!gh\!-\!fh)$},\\[3pt]
$H_2\cong R$ generated by $\{efg\!-\!efh\!+\!egh\!-\!fgh\}$,\\[6pt]
$H^0\cong R,\;\; H^1\cong R\!\oplus\!\frac{R}{Rt^3}\!\oplus\!\frac{R}{Rt}\!\oplus\!\frac{R}{Rt},\;\; H^2\cong R\!\oplus\!\frac{R}{Rt^3}\!\oplus\!\frac{R}{Rt^2}\!\oplus\!\frac{R}{Rt}$.
\end{longtable}\vspace{2mm}

\subsection{Interpretation} The indeterminate $t\!\in\!K[t]$ represents 'time', so that $\Delta_t$ grows as $t$ passes. If $\sigma\!\in\!\Delta_k$, then $t\sigma\!\in\!\Delta_{k+1}$, so multiplication by $t^l$ means going $l$ steps up in the filtration. The complex $P_\ast(\Delta_\ast)$ is finite and each of its modules has finite rank, so by the classification theorem \cite[8.6.3, p.339]{citeGrilletAA}, each persistent (co)homology is isomorphic to $R^r\!\oplus\!R/Rf_1\!\oplus\!\ldots\!\oplus\!R/Rf_s$ for some $f_1,\ldots,f_s\!\in\!K[t]$. Since $\deg t^{\deg\!\sigma-\deg\!\tau}\tau=\deg\sigma\!-\!\deg\tau\!+\!\deg\tau=\deg\sigma$, all maps $\partial_1,\ldots,\partial_N$ are graded module morphisms, so $\frac{\Ker\partial_n}{\Im\partial_{n+1}}$ is a graded module, hence each persistent (co)homology is isomorphic to $R^r\!\oplus\!R/Rt^{l_1}\!\oplus\!\ldots\!\oplus\!R/Rt^{l_s}$ for some $l_1,\ldots,l_s\!\in\!\N$. A generator of $R^r$ has \emph{lifetime $\infty$}, because multiplying with any $t^k$ does not annihilate it. A generator of $R/Rt^{l_i}$ has \emph{lifetime $l_i$}, because multiplying with $t^{l_i}$ annihilates it, meaning the `hole' is present for $l_i$ steps of the filtration and then gets `filled up'.
\par If $N\!=\!0$, then persistent (co)homology is the usual simplicial (co)homology of $\Delta$. Notice that persistent (co)homology computes simplicial (co)homology of \emph{every} $\Delta_k$, even though the persistence chain complex for $\Delta_\ast$ and Poincar\'{e} chain complex for $\Delta$ have the same bases (the additional information for the former is encoded in all the coefficients $t^k$). The (co)homology of $\Delta_N$ is precisely the part with infinite lifetime of the persistent (co)homology of $\Delta_\star$.\vspace{2mm}

\subsection{Generalization} The above construction is also applicable for other homological theories. Let $K$ be a field and $(C_\star,\partial_\star)$ a chain complex of $K$\!-modules in which every $C_n$ admits a basis $B_n$ that is equipped with a filtration $B_{n,0}\!\subseteq\!B_{n,1}\!\subseteq\!\ldots\!\subseteq\!B_{n,N}\!=\!B_n$. This induces a grading $C_n\!=\!\bigoplus_{k\geq0}C_{n,k}$, where $C_{n,k}$ is the free module on $B_{n,k}\!\setminus\!B_{n,k\!-\!1}$, so $C_n$ is the associated graded module of the filtration. If every boundary map is written with respect to the bases $B_n$ and $B_{n-1}$, i.e $\partial_n(b')=\sum_{b\in B_{n\!-\!1}}[b',b]b$ for $b'\!\in\!B_n$, then defining $$\textstyle{\overline{\partial}_n(b')=\sum_{b\in B_{n\!-\!1}}[b',b]\,t^{\deg\!b'-\deg\!b}\,b}$$ makes $(C_\star\!\otimes\!_K\!R,\overline{\partial}_\star)$ a chain complex of graded $R$\!-modules. In this way, we obtain the \emph{persistent} (co)homology of $(C_\star,\partial_\star)$.
\par Let us observe this situation on the three most popular algebraic homology theories. If $G_\star: G_0\!\leq\!G_1\!\leq\!\ldots\!\leq\!G_N$ is an increasing sequence of (semi)groups, then their Cartesian powers also constitute $G_0^n\!\leq\!G_1^n\!\leq\!\ldots\!\leq\!G_N^n$, so this induces \emph{persistent} Eilenberg-MacLane (co)homology of $G_\star$. The symmetric groups, Braid groups, general/special linear groups, $\ldots$ are examples of $G_\star$. If $A_\star: A_0\!\leq\!A_1\!\leq\!\ldots\!\leq\!A_N$ is an increasing sequence of associative $K$-algebras, then their tensor powers also constitute $T^n\!A_0\!\leq\!T^n\!A_1\!\leq\!\ldots\!\leq\!T^n\!A_N$, so this induces \emph{persistent} Hochschild (co)homology of $A_\star$. The matrix algebras, (non)commutative polynomial algebras, Laurent polynomial algebras, formal power series algebras, $\ldots$ are examples of $A_\star$. If $\frak{g}_\star: \frak{g}_0\!\leq\!\frak{g}_1\!\leq\!\ldots\!\leq\!\frak{g}_N$ is an increasing sequence of Lie $K$-algebras, then their exterior powers also constitute $\Lambda^{\!n}\frak{g}_0\!\leq\!\Lambda^{\!n}\frak{g}_1\!\leq\!\ldots\!\leq\!\Lambda^{\!n}\frak{g}_N$, so this induces \emph{persistent} Chevalley (co)homology of $\frak{g}_\star$. The general/special linear Lie algebras, solvable Lie algebras, nilpotent Lie algebras, $\ldots$ are examples of $\frak{g}_\star$.

\vspace{4mm}
\section{Module powers $T^n$\!, $S^n$\!, $\Lambda^{\!n}$}\label{3.TnSnEn}
All tensor products are over a commutative unital ring $R$. If $R$ is a PID, then the ideal $\langle a_1,\ldots,a_s\rangle$ is generated by $\gcd(a_1,\ldots,a_s)$. If $R\!=\!K[t]$ and $a_i\!=\!t^{l_i}\!\in\!R$ for all $i$, then $\gcd(a_1,\ldots,a_k) \!=\! t^{\min(l_1,\ldots,l_k)}$.\vspace{2mm}

\subsection{Tensor Powers} The \emph{$n$-th tensor power} of an $R$-module $M$ is $T^n\!M=\bigotimes_{i=1}^n\!M$, the tensor product of $n$ copies of $M$.\vspace{1mm}

\begin{Prp}\label{Tn}
If $M\cong R^r\!\oplus\!\bigoplus_{i=1}^s\!R/\!Ra_i$, then\vspace{-1mm}
$$T^n\!M\cong R^{r^n}\!\oplus\bigoplus_{k=1}^n \bigoplus_{1\leq i_1,\ldots,i_k\leq s}\! \big(R/\!\langle a_{i_1},\ldots,a_{i_k}\rangle\big)^{\binom{n}{k}r^{n\!-\!k}}.$$
\end{Prp}
\begin{proof}
By definition, we have $T^0(M)\!=\!R$ and $T^1(M)\!=\!M$. There are isomorphisms $$R^r\!\otimes\!R^s\cong R^{rs},~~~ R^r\otimes R/\!Rb\cong (R/\!Rb)^r,~~~ R/\!Ra\otimes R/\!Rb\cong R/\!R\gcd(a,b).$$
Hence we have $T^n(R^r)\cong R^{r^n}$. For the general case, since tensor products are distributive w.r.t. direct sums, $T^n\!M$ is a direct sum of modules $M_1\!\otimes\!\ldots\!\otimes\!M_n$ in which every $M_i\in\{R^r\!,R/Ra_1,\ldots,R/Ra_s\}$. If such a module contains $n\!-\!k$ copies of $R^r$ and the rest is $R/Ra_{i_1},\ldots,R/Ra_{i_k}$ in some order (different reorderings count as distinct modules), then the commutativity isomorphisms imply that it is isomorphic to $R^r\!\otimes\!\ldots\!\otimes\!R^r \!\otimes\! R/Ra_{i_1}\!\otimes\!\ldots\!\otimes\!R/Ra_{i_k} \cong \big(R/\langle a_{i_1},\ldots,a_{i_k}\rangle\big)^{r^{n\!-\!k}}$. There are $\binom{n}{n-k}$ choices where $R^r$ are positioned, so there are $\binom{n}{k}$ copies of the latter module.
\end{proof}\vspace{2mm}

\subsection{Symmetric Powers} The \emph{$n$-th symmetric power} of an $R$-module $M$ is
$$\textstyle{S^n\!M=T^n\!M/\langle m_{\pi1}\!\otimes\!\ldots\!\otimes\!m_{\pi n}\!-\!m_1\!\otimes\!\ldots\!\otimes\!m_n;\,m_1,\ldots,m_n\!\in\!M, \pi\!\in\!S_n\rangle.}\vspace{1mm}$$

\begin{Prp}
If $M\cong R^r\!\oplus\!\bigoplus_{i=1}^s\!R/\!Ra_i$, then\vspace{-1mm} $$S^n\!M\cong R^{\binom{r+n\!-\!1}{n}}\!\oplus\bigoplus_{k=1}^n \bigoplus_{1\leq i_1\leq\ldots\leq i_k\leq s}\! \big(R/\!\langle a_{i_1},\ldots,a_{i_k}\rangle\big)^{\binom{r+n\!-\!k\!-\!1}{n\!-\!k}}.$$
\end{Prp}
\begin{proof}
By definition, we have $S^0(M)\!=\!R$ and $S^1(M)\!=\!M$. There holds $$\textstyle{S^n(R/I)\cong R/I\text{ for }n\!\geq\!2 ~~~~\text{ and }~~~~ S^n(A\!\oplus\!B)\cong \bigoplus_{i+j=n}\! S^{i}\!(A)\!\otimes\!S^{j}\!(B).}$$
The isomorphism $R\!\otimes\!\ldots\!\otimes\!R\!\rightarrow\!R$ that sends $x_1\!\otimes\!\ldots\!\otimes\!x_r\mapsto x_1\!\cdots\!x_r$, with inverse $x\!\mapsto\!x\!\otimes\!1\!\otimes\!\ldots\!\otimes\!1$, descends to an isomorphism $\frac{R}{I}\!\otimes\!\ldots\!\otimes\!\frac{R}{I}\!\rightarrow\!\frac{R}{I}$ (if any of the $x_i$ is in $I$, then so is $x_1\!\cdots\!x_r$) and to an isomorphism $S^n\big(\tfrac{R}{I}\big)\!\rightarrow\!\frac{R}{I}$ (the permutation relation is sent to $0$ since $R$ is commutative), so this justifies the first claim. The map $\bigoplus_{i+j=n}\! S^{i}\!(A)\!\otimes\!S^{j}\!(B)\longrightarrow S^n(A\!\oplus\!B)$ that sends $(a_1\!\otimes\!\ldots\!\otimes\!a_i)\!\otimes\!(b_1\!\otimes\!\ldots\!\otimes\!b_j)\mapsto a_1\!\otimes\!\ldots\!\otimes\!a_i\!\otimes\!b_1\!\otimes\!\ldots\!\otimes\!b_j$ is well defined (since $S_n$ is closed for concatenation) and surjective (since $\otimes$ are distributive and in $S^n\!M$ every pure tensor can be permuted) with inverse $(a_1\!+\!b_1)\!\otimes\!\ldots\!\otimes\!(a_n\!+\!b_n) \mapsto \sum_{k+l=n}(a_{i_1}\!\!\otimes\!\ldots\!\otimes\!a_{i_k})\!\otimes\!(b_{j_1}\!\!\otimes\!\ldots\!\otimes\!b_{j_l}\!)$.
\par Then inductively we get $S^n(M_1\!\oplus\!\ldots\!\oplus\!M_k)\cong \bigoplus_{i_1+\ldots+i_k=n}\! S^{i_1}\!(M_1)\!\otimes\!\ldots\!\otimes\!S^{i_k}\!(M_k)$. Therefore $S^n(R^r)\cong \bigoplus_{i_1+\ldots+i_r=n}\!R\cong R^{\binom{r+n\!-\!1}{n}}$ where $\binom{r+n-1}{n}\!=\!\big(\!\!\binom{r}{n}\!\!\big)$ is the number of submultisets in $[r]$ of cardinality $n$. Furthermore,
\begin{longtable}[c]{r@{ $\cong$ }l}
    $S^n\!\big(R^r\!\oplus\!\bigoplus_{i=1}^s\!\tfrac{R}{Ra_i}\big)$
    & $\bigoplus_{i_0+i_1+\ldots+i_s=n}\! S^{i_0}\!(R^r)\!\otimes\!S^{i_1}\!\big(\tfrac{R}{Ra_1}\big)\!\otimes\!\ldots\!\otimes\!S^{i_k}\!\big(\tfrac{R}{Ra_s}\big)$\\
    & $\bigoplus_{k=0}^n\!\bigoplus_{i_1+\ldots+i_s=k}\! S^{n\!-\!k}\!(R^r)\!\otimes\!S^{i_1}\!\big(\!\tfrac{R}{Ra_1}\!\big)\!\otimes\!\ldots\!\otimes\!S^{i_k}\!\big(\!\tfrac{R}{Ra_s}\!\big)$\\
    & $R^{\left(\!\!\binom{r}{n}\!\!\right)}\!\oplus\!\bigoplus_{k=1}^n\!\bigoplus_{i_1+\ldots+i_s=k}\! \big(S^{i_1}\!\big(\!\tfrac{R}{Ra_1}\!\big)\!\otimes\!\ldots\!\otimes\!S^{i_k}\!\big(\!\tfrac{R}{Ra_s}\!\big)\!\big)^{\!\left(\!\!\binom{r}{n\!-\!k}\!\!\right)}$\\
    & $R^{\left(\!\!\binom{r}{n}\!\!\right)}\!\oplus\!\bigoplus_{k=1}^n\!\bigoplus_{\sigma\in\left(\!\!\binom{[s]}{k}\!\!\right)}\! \big(R/\langle a_i; i\!\in\!\sigma\rangle\!\big)^{\!\left(\!\!\binom{r}{n\!-\!k}\!\!\right)}$,
\end{longtable}\vspace{-3mm}
where $\big(\!\!\binom{[s]}{k}\!\!\big)$ is the set of all $k$-element multisubsets of $\{1,\ldots,s\}$.
\end{proof}\vspace{2mm}

\subsection{Exterior Powers} The \emph{$n$-th exterior power} of an $R$-module $M$ is
$$\textstyle{\Lambda^{\!n}\!M=T^n\!M/\langle m_1\!\otimes\!\ldots\!\otimes\!m_n;\,m_1,\ldots,m_n\!\in\!M, m_i\!=\!m_j\text{ for some }i\!\neq\!j\rangle.}\vspace{1mm}$$

\begin{Prp}
If $M\cong R^r\!\oplus\!\bigoplus_{i=1}^s\!R/\!Ra_i$, then\vspace{-1mm}
$$\Lambda^{\!n}\!M\cong R^{\binom{r}{n}}\!\oplus\bigoplus_{k=1}^n \bigoplus_{1\leq i_1<\ldots<i_k\leq s}\! \big(R/\!\langle a_{i_1},\ldots,a_{i_k}\rangle\big)^{\binom{r}{n\!-\!k}}.$$
\end{Prp}
\begin{proof}
By definition, we have $\Lambda^{\!0}(M)\!=\!R$ and $\Lambda^{\!1}(M)\!=\!M$. There holds $$\textstyle{\Lambda^{\!n}(R/I)\cong 0\text{ for }n\!\geq\!2 ~~~~\text{ and }~~~~ \Lambda^{\!n}(A\!\oplus\!B)\cong \bigoplus_{i+j=n}\! \Lambda^{\!i}\!(A)\!\otimes\!\Lambda^{\!j}\!(B).}$$
If $n\!\geq\!2$, then any pure tensor $x_1\!\otimes\!\ldots\!\otimes\!x_n\!\in\!\Lambda^{\!n}(R/I)$ is equal to the element $1\!\otimes\!\ldots\!\otimes\!1\!\otimes\!x_1\!\cdots\!x_n$, so it is contained in the submodule of relations, which shows that $\Lambda^{\!n}(R/I)\cong0$. The morphism $\bigoplus_{i+j=n}\! \Lambda^{\!i}\!(A)\!\otimes\!\Lambda^{\!j}\!(B)\longrightarrow \Lambda^{\!n}(A\!\oplus\!B)$ that sends $(a_1\!\otimes\!\ldots\!\otimes\!a_i)\!\otimes\!(b_1\!\otimes\!\ldots\!\otimes\!b_j)\mapsto a_1\!\otimes\!\ldots\!\otimes\!a_i\!\otimes\!b_1\!\otimes\!\ldots\!\otimes\!b_j$ is well defined (any duplicate in $\Lambda^{\!i}\!(A)$ or $\Lambda^{\!j}\!(B)$ leads to a duplicate in $\Lambda^{\!n}(A\!\oplus\!B)$) and surjective (since $\otimes$ are distributive and in $\Lambda^{\!n}M$ every pure tensor can be permuted up to sign) with inverse $(a_1\!+\!b_1)\!\otimes\!\ldots\!\otimes\!(a_n\!+\!b_n) \mapsto \sum_{k+l=n}(a_{i_1}\!\!\otimes\!\ldots\!\otimes\!a_{i_k})\!\otimes\!(b_{j_1}\!\!\otimes\!\ldots\!\otimes\!b_{j_l}\!)$.
\par Then inductively we get $\Lambda^{\!n}(M_1\!\oplus\!\ldots\!\oplus\!M_k)\cong \bigoplus_{i_1+\ldots+i_k=n}\! \Lambda^{\!i_1}\!(M_1)\!\otimes\!\ldots\!\otimes\!\Lambda^{\!i_k}\!(M_k)$. Therefore $\Lambda^{\!n}(R^r)\cong \bigoplus_{i_1,\ldots,i_r\in\{0,1\},\,i_1+\ldots+i_r=n}\!R\cong R^{\binom{r}{n}}$. Furthermore,
\begin{longtable}[c]{r@{ $\cong$ }l}
    $\Lambda^{\!n}\!\big(R^r\!\oplus\!\bigoplus_{i=1}^s\!\tfrac{R}{Ra_i}\big)$
    & $\bigoplus_{i_0+i_1+\ldots+i_s=n}\! \Lambda^{\!i_0}\!(R^r)\!\otimes\!\Lambda^{\!i_1}\!\big(\tfrac{R}{Ra_1}\big)\!\otimes\!\ldots\!\otimes\!\Lambda^{\!i_s}\!\big(\tfrac{R}{Ra_s}\big)$\\
    & $\bigoplus_{k=0}^n\!\bigoplus_{i_1+\ldots+i_s=k}\! \Lambda^{\!n\!-\!k}\!(R^r)\!\otimes\! \Lambda^{\!i_1}\!\big(\!\tfrac{R}{Ra_1}\!\big)\!\otimes\!\ldots\!\otimes\!\Lambda^{\!i_s}\!\big(\!\tfrac{R}{Ra_s}\!\big)$\\
    & $R^{\binom{r}{n}}\!\oplus\!\bigoplus_{k=1}^n\!\bigoplus_{i_1+\ldots+i_s=k}\! \big(\Lambda^{\!i_1}\!\big(\!\tfrac{R}{Ra_1}\!\big)\!\otimes\!\ldots\!\otimes\!\Lambda^{\!i_s}\!\big(\!\tfrac{R}{Ra_s}\!\big)\!\big)^{\!\binom{r}{n\!-\!k}}$\\
    & $R^{\binom{r}{n}}\!\oplus\!\bigoplus_{k=1}^n\!\bigoplus_{\sigma\in\binom{[s]}{k}}\! \big(R/\langle a_i; i\!\in\!\sigma\rangle\!\big)^{\!\binom{r}{n\!-\!k}}$,
\end{longtable}\vspace{-3mm}
where $\binom{[s]}{k}$ is the set of all $k$-element subsets of $\{1,\ldots,s\}$.
\end{proof}\vspace{2mm}

\subsection{Graded Algebras} The graded $R$-modules $$\textstyle{T(M)=\bigoplus_{n\in\N}T^n(M),\; S(M)=\bigoplus_{n\in\N}S^n(M),\; \Lambda(M)=\bigoplus_{n\in\N}\Lambda^{\!n}(M)}$$ become associative unital graded $R$-algebras via $$(x_1\!\otimes\!\ldots\!\otimes\!x_i)\cdot(y_1\!\otimes\!\ldots\!\otimes\!y_j)=x_1\!\otimes\!\ldots\!\otimes\!x_i \!\otimes\! y_1\!\otimes\!\ldots\!\otimes\!y_j.$$

\begin{Prp}
If $M\cong R^r\!\oplus\!\bigoplus_{i=1}^s\!R/\!Ra_i$, then
\begin{longtable}[c]{c}
$T(M)\cong R\langle x_1,\ldots,x_r,y_1,\ldots,y_s|a_1y_1,\ldots,a_sy_s\rangle$,\\
$S(M)\cong R[x_1,\ldots,x_r,y_1,\ldots,y_s|a_1y_1,\ldots,a_sy_s]$,\\
$\Lambda(M)\cong \Lambda[x_1,\ldots,x_r,y_1,\ldots,y_s|a_1y_1,\ldots,a_sy_s]$.\\
\end{longtable}\vspace{-2mm}
\end{Prp}
These are quotients of noncommutative/commutative/anticommutative polynomial algebras by a two-sided ideal generated by all $a_iy_i$.
\begin{proof} If $M\!=\!R$ or $M\!=\!\frac{R}{Ra}$, then we have isomorphisms of $R$-modules
$$T(M)\!\cong\!S(M)\!\cong\!R\!\oplus\!M\!\oplus\!M\!\oplus\!M\!\oplus\!\ldots\text{ and }\Lambda(M)\!\cong\!R\!\oplus\!M.$$
The generator in each degree is $1^{\otimes n}\!\equiv\!1\!\otimes\!\ldots\!\otimes\!1$, so $1^{\otimes i}\!\cdot\!1^{\otimes j}=1^{\otimes i+j}$, and we have
$$T(R)\!\cong\!S(R)\!\cong\!R[x],\; \Lambda(R)\cong\Lambda[x],$$ $$T(\tfrac{R}{Ra})\!\cong\!S(\tfrac{R}{Ra})\!\cong\!R[x|ax],\; \Lambda(\tfrac{R}{Ra})\cong\Lambda[x|ax].$$
Then the properties of free/tensor/skew-tensor products imply\vspace{-1mm}
\begin{longtable}[c]{c}
$T(M)\cong \bigast_{i=1}^rT(R)\ast\bigast_{i=1}^sT(\tfrac{R}{Ra_i})$,\\
$S(M)\cong \bigotimes_{i=1}^rT(R)\otimes\bigotimes_{i=1}^sT(\tfrac{R}{Ra_i})$,\\
$\Lambda(M)\cong {\bigotimes\!'}_{\!\!i=1}^rT(R)\otimes'{\bigotimes\!'}_{\!\!i=1}^sT(\tfrac{R}{Ra_i})$.\\
\end{longtable}\vspace{-1mm}
The result then follows from the fact that $R\langle x_i|f_j\rangle\ast R\langle y_k|g_l\rangle \cong R\langle x_i,y_k|f_j,g_l\rangle$ and $R[ x_i|f_j]\otimes R[ y_k|g_l] \cong R[ x_i,y_k|f_j,g_l]$ and $\Lambda[ x_i|f_j]\ast \Lambda[ y_k|g_l] \cong \Lambda[ x_i,y_k|f_j,g_l]$.
\end{proof}\vspace{2mm}

\subsection{Interpretations} The following descriptions will be very vague, I am sorry. Let us imagine a persistence module $M=H_k\big(P_\ast(\Delta_\ast)\big)$ (i.e. the $k$-th homology of the persistence complex for $\Delta_\ast$ over $R\!=\!K[t]$) as a `system of $k$-dimensional holes' that `vary through time'. If the ground field $K$ has nonzero characteristic, then `hole' has a broad meaning. An element $x\!\in\!M$ is a $K$-linear combination of holes (let's call it a `cave'), and over time $x,tx,t^2x,t^3x,\ldots$ some of its parts get filled in. An element of the Cartesian power $\prod^n\!M\!=\!\bigoplus^n\!M$ corresponds to the disjoint union of $n$ caves: they are independent of each other, since multiplying by $t$ affects each component separately.
\par Recall that all our tensor products are over $R$, not just over $K$. A pure element $x_1\!\otimes\!\ldots\!\otimes\!x_n\in T^nM$ is a sequence of $n$ caves, but now they affect each other, since $(tx)\!\otimes\!y\!=\!x\!\otimes\!(ty)$. If $x$ dies after $i$ steps and $y$ was born more than $i$ steps ago, then $x\!\otimes\!y\!=\!0$ is filled in. If $x_i\!=\!0$ is filled in, then so is the whole sequence $x_1\!\otimes\!\ldots\!\otimes\!x_i\!\otimes\!\ldots\!\otimes\!x_n$. A pure element in $S^n$ is a multisubset of $n$ caves, and a pure element of $\Lambda^{\!n}$ is a subset of $n$ caves: again each cave affects the other, but now their ordering is not important. Multiplication in $T(M),S(M),\Lambda(M)$ means we take two sequences/multisets/sets and make them dependent on each other.

\vspace{4mm}
\section{Other module powers}
In this section, we are interested in the quotients of $T^n\!M$ by a given group action, where $M\!=\!R^r\!\oplus\!\bigoplus_{i=1}^sR/Ra_i$.\\[2mm]

\noindent The \emph{$n$-th cyclic power} of $M$, which appears e.g. in \cite[p.53]{citeLodayCH}, is $$T^n_C(M)=T^n(M)/\langle m_1\!\otimes\!m_2\!\otimes\!\ldots\!\otimes\!m_n \!-\! m_2\!\otimes\!\ldots\!\otimes\!m_n\!\otimes\!m_1\rangle.$$
The \emph{$n$-th dihedral power} of $M$, which appears e.g. in \cite[p.169]{citeLodayCH}, is $$T^n_D(M)=T^n(M)/\big\langle
\begin{smallmatrix}m_1\otimes m_2\otimes\ldots\otimes m_n - m_2\otimes\ldots\otimes m_n\otimes m_1,\\
m_1\otimes m_2\otimes\ldots\otimes m_n - m_n\otimes\ldots\otimes m_2\otimes m_1\end{smallmatrix}\big\rangle.$$
They are used (with additional sign conventions) to define the cyclic and dihedral (co)homology of associative algebras.
\par More generally, let $G$ be a finite group with a given action on $[n]$, or equivalently, let $G\!\leq\!S_n$ be a subgroup of the symmetric group. The \emph{$n$-th $G$-power} of $M$ is $$T^n_G(M)=T^n(M)/\big\langle m_{\pi1}\!\otimes\!\ldots\!\otimes\!m_{\pi n}\!-\!m_1\!\otimes\!\ldots\!\otimes\!m_n;\,m_1,\ldots,m_n\!\in\!M, \pi\!\in\!G\big\rangle.$$
Notice that the submodule of all relations equals the submodule generated by relations where $\pi$ runs through all generators of $G$ instead of all elements of $G$. For instance, the relations of the symmetric power $T^n_S(M)\!=\!S^n(M)$ are generated by the relations obtained from adjacent transpositions $\pi\!=\!(i,i\!+\!1)$ for $1\!\leq\!i\!<\!n$.\\[2mm]

\par Methods from the proofs in \ref{3.TnSnEn} do not work: the map $T^i_C(A)\!\otimes\!T^j_C(B)\!\longrightarrow\!T^n_C(A\!\oplus\!B)$ that sends $(a_1\!\otimes\!\ldots\!\otimes\!a_i) \!\otimes\! (b_1\!\otimes\!\ldots\!\otimes\!b_j) \!\longmapsto\! a_1\!\otimes\!\ldots\!\otimes\!a_i \!\otimes\! b_1\!\otimes\!\ldots\!\otimes\!b_j$ is not well defined, since it doesn't respect the cyclic relations. However, we still have $$T^n_G(R/I)\cong R/I$$ for any ideal $I\!\unlhd\!R$ and $n\!\geq\!1$, because the map $x_1\!\otimes\!\ldots\!\otimes\!x_n \mapsto x_1\!\cdots\!x_n$ respects all relations of $G$ since $R$ is commutative. Also, $$T^n_C(R^r)\cong R^{c_{r,n}} \;\;\text{ and }\;\; T^n_D(R^r)\cong R^{d_{r,n}}$$ where $c_{r,n}=\frac{1}{n}\sum_{d|n}\phi(d)r^{n\!/\!d}$ and $d_{r,n}=\icases{c_{r,n}/2+(r+1)r^{n/2}/4}{n\text{ even}}{c_{r,n}/2+r^{(n+1)/2}/2}{n\text{ odd}}{\scriptstyle}{-1pt}{\Big}$ are the number of necklaces and bracelets. In general, $G\!\leq\!S_n$ acts on $[n]$, hence it acts on $[r]^{[n]}$ (the basis elements of $T^n\!R^r$) via $(\pi.f)(i)\!=\!f(\pi.i)$, and then $$T^n_G(R^r)\cong R^{g_{r,n}}$$ where $g_{r,n}$ is the number of orbits of the action of $G$ on $[r]^{[n]}$. Burnside's lemma tells us that if a group $G$ acts on a set $X$, then the number of orbits times the size of $G$ equals the number of all fixed points: $$\textstyle{|X/G|=\frac{1}{|G|}\sum_{g\in G}|X^g|}$$ where $X/G\!=\!\{[x];\, x\!\in\!X, x\!\sim\!x'\text{ iff }\exists g\!\in\!G\!:g.x\!=\!x'\}$ and $X^g\!=\!\{x\!\in\!X;\, g.x\!=\!x\}$. If $G\!\leq\!S_n$ and $X\!=\![r]^{[n]}$, we have $\pi.f\!=\!f$ iff $f$ is constant on the orbits of $\pi$, so $|X^\pi|\!=\!r^{|[n]/\langle\pi\rangle|}$ and the above formula can be applied to compute $g_{r,n}=|X/G|$.  More generally, we have $T^n\big((R/I)^r\big)\cong(R/I)^{g_{r,n}}$ for any ideal $I$ of $R$. \\[2mm]

\par Now let $M\!=\!R^r\!\oplus\!\bigoplus_{i=1}^sR/Ra_i$ be arbitrary. If we allow the elements $a_i$ to be zero, then we do not lose any generality by assuming $M\!=\!\bigoplus_{i=1}^sR/Ra_i$. The $R$-module $T^n\!M$ has a minimal set of generators
$$\big\{x_1\!\otimes\!\ldots\!\otimes\!x_n;\, x_1,\ldots,x_n \!\in\! \{1\,\mod\,a_1,\ldots,1\,\mod\,a_s\}\big\},$$
where $(1\,\mod\,a_{i_1})\!\otimes\!\ldots\!\otimes\!(1\,\mod\,a_{i_n})$ spans the direct summand $R/\langle a_{i_1},\ldots,a_{i_n}\rangle$ (see the proof of \ref{Tn}). These generators  correspond to maps $f\!:[n]\!\rightarrow\![s]$ via the rule $i_1\!=\!f(1),\ldots,i_n\!=\!f(n)$. The relations induced by $G$ that produce $T^n_GM$ out of $T^n\!M$ act on the above set of generators, i.e. $G$ acts on $[s]^{[n]}$. Since in general $\frac{R}{\langle a_{i_1}\!,\ldots,a_{i_n}\rangle}\!\oplus\!\frac{R}{\langle b_{j_1}\!,\ldots,b_{j_m}\rangle}/\langle (1,0)\!-\!(0,1)\rangle \cong\frac{R}{\langle a_{i_1}\!,\ldots,a_{i_n}\!, b_{j_1}\!,\ldots,b_{j_m}\rangle}$, we obtain the formula 
$$\textstyle{T^n_G\:\cong\, \bigoplus_{F\in[s]^{[n]}\!/G} R/\langle a_{f(i)};\, i\!\in\![n], f\!\in\!F\rangle.}$$
Here $F$ runs through all orbits of the action of $G$ on $[s]^{[n]}$. Among $a_{f(i)}$ the zero elements can be removed from the submodule, and then that submodule is generated by $\gcd$ of all remaining $a_{f(i)}$ when $R$ is a B\'{e}zout domain.\vspace{4mm}

\section{Conclusion} For any module of the form $M=R^r\!\oplus \bigoplus_{i=1}^s\!R/Ra_i$, we can compute the modules $T^n\!M, S^n\!M, \Lambda^{\!n}\!M$ and $T^n_GM$ directly from the relations $a_1,\ldots,a_n$, no time consuming matrix computations are needed. When $M$ is a persistence module, this enables us to make combinatorial estimations on the size and number of all the barcodes in the various tensor powers of $M$.
\end{document}